\documentclass[a4paper,english,smallextended]{article}
\usepackage[T1]{fontenc}
\usepackage[latin1]{inputenc}
\usepackage{prettyref}
\usepackage{amsthm}
\usepackage{amsmath}
\usepackage{amssymb}

\makeatletter

\theoremstyle{plain}
\newtheorem{thm}{\protect\theoremname}
  \theoremstyle{plain}
  \newtheorem{prop}[thm]{\protect\propositionname}
  \theoremstyle{plain}
  \newtheorem{cor}[thm]{\protect\corollaryname}
  \theoremstyle{plain}
  \newtheorem{lem}[thm]{\protect\lemmaname}
  \theoremstyle{remark}
  \newtheorem{rem}[thm]{\protect\remarkname}

\date{}

\makeatother

\usepackage{babel}
  \providecommand{\corollaryname}{Corollary}
  \providecommand{\lemmaname}{Lemma}
  \providecommand{\propositionname}{Proposition}
  \providecommand{\remarkname}{Remark}
\providecommand{\theoremname}{Theorem}

\begin{document}

\title{Multiset metrics on bounded spaces%
\thanks{MSC primary 51F99,03E70; secondary 54E50,62H30,68T10,91B12,92C42.%
}}

\author{Stephen M. Turner}
\maketitle
\begin{abstract}
We discuss five simple functions on finite multisets of metric spaces.
The first four are all metrics iff the underlying space is bounded
and are complete metrics iff it is also complete. Two of them, and
the fifth function, all generalise the usual Hausdorff metric on subsets.
Some possible applications are also considered.
\end{abstract}

\section{Introduction}

Metrics on subsets and multisets (subsets-with-repetition-allowed)
of metric spaces have or could have numerous fields of application
such as credit rating, pattern or image recognition and synthetic
biology. We employ three related models (called $E,F$ and $G$) for
the space of multisets on the metric space $(X,d)$. On each of $E,F$
we define two closely-related functions. These four functions all
turn out to be metrics precisely when $d$ is bounded, and are complete
iff $d$ is also complete. Another function studied in model $G$
has the same properties for (at least) uniformly discrete $d$. $X$
is likely to be finite in many applications anyway.

We show that there is an integer programming algorithm for those in
model $E$. The three in models $F$ and $G$ generalise the Hausdorff
metric. Beyond finiteness, no assumptions about multiset sizes are
required.

Various types of multiset metric on sets have been described{[}\prettyref{sub:Other-metrics}{]},
but the few that incorporate an underlying metric only refer to $\mathbb{R}$
or $\mathbb{C}$. The simple and more general nature of those described
here suggests that there may be other interesting possibilities.

In this section, after setting out notation and required background,
we mention briefly the existing work in this field. The following
three sections are each dedicated to one of $E$, $F$ and $G$.

\subsection{Notation: metric spaces}

$R$ is the non-negative reals, $\mathbb{N}$ includes $0$, and $(X,d)$
is a metric space of more than one element. $d$ is \emph{uniformly
discrete} if $\exists a>0$ such that $d(x,z)\geq a$ whenever $x\neq z$,
and two metrics on $X$ are \emph{equivalent} if they induce the same
topology.

$d$ is \emph{complete} iff every Cauchy sequence converges (to a
point of $X$), and $d$ is \emph{compact} iff every sequence, Cauchy
or not, has a subsequence that converges to a point of $X$.

The well-known \emph{Hausdorff metric} $d_{H}$ on the space $H$
of all non-empty compact subsets of $X$ is defined for $A,B\in H$
by 
\[
d_{H}(A,B)=\max(\max_{x\in A}\min_{y\in B}d(x,y),\max_{y\in B}\min_{x\in A}d(x,y))
\]

in which compactness guarantees that all these extrema are attained.
We use later the simple fact that if $d$ does not satisfy the triangle
inequality, neither does $d_{H}$. It is a standard fact\cite[pp.71-72]{Edg90}
that $d_{H}$ is complete if $d$ is. The converse is also true%
\footnote{Let $x_{i}$ be a non-convergent Cauchy sequence in $X$ so that $S_{i}=\{x_{i}\}$
is Cauchy in $H$ with putative limit $S\in H$, so $S$ is non-empty.
If $S=\{x\}$ then $d(x_{i},x)\rightarrow0$. Thus $S$ contains distinct
$a,b\in X$. But then $d_{H}(S_{i},S)\geq\max(d(x_{i},a),d(x_{i},b))\geq\frac{d(x_{i},a)+d(x_{i},b)}{2}\geq\frac{d(a,b)}{2}$.%
}. A convenient heuristic (for finite $A,B$) is to label the rows
(the columns) of a matrix by the elements of $A$ (of $B$), with
the corresponding $d$-distances as entries. Then $d_{H}(A,B)$ is
the largest of all row and column minima. 

\label{quotient}Given an equivalence relation $\sim$ on $X$, and
$\alpha,\beta\in X/\!\sim$ write 
\[
D(\alpha,\beta)=\inf d(a,p_{1})+d(q_{1},p_{2})+d(q_{2},p_{3})+\ldots+d(q_{n-1},p_{n})+d(q_{n},b)
\]

where $n\in\mathbb{N}$, $a\in\alpha$, $b\in\beta$ and $p_{i}\sim q_{i}$
for each $i$. In general $D$ is a \emph{pseudometric} on $X/\!\sim$,
that is $D(\alpha,\beta)=0\nRightarrow\alpha=\beta$, though $D$
does satisfy the other metric axioms. Clearly $D(\alpha,\beta)\leq\inf_{a\in\alpha,b\in\beta}d(a,b)$.
To simplify notation, we adopt the conventions that $a=q_{0},b=p_{n+1}$
and $p_{i}\nsim p_{i+1}$ for any $i$.

\subsection{Notation: multisets}

A recent survey article on multisets and their applications is \cite{SIYS07}.
The notation and terminology in this article mostly follow \cite{DD09}
and \cite{Pet97}. A convenient definition of multiset also introduces
the model $E${[}\prettyref{sec:Model E}{]}.\emph{ }

A \emph{multiset} of a set $S$ is a function $e:S\mathbb{\rightarrow N}$
taking each $s\in S$ to its \emph{multiplicity} $e(s)$. The \emph{root
set} $R(e)$ of $e$ is $\{s\in S:e(s)>0\}$, always assumed finite.
The \emph{cardinality}%
\footnote{Called the \emph{counting measure} in \cite{DD09}.%
} of $e$ is $C(e)=\sum_{s\in S}e(s)$. So $E$ is the set of functions
of finite support from $S$ to $\mathbb{N}$. 

We denote by $e_{s}$, for $s\in S$, the multiset consisting of a
single copy of $s$ and define $e_{0}$ by $R(e_{0})=\phi$. Naturally
any multiset has a unique form $\sum_{s\in S}e(s)e_{s}$; we can add
or subtract them if all the arithmetic is within $\mathbb{N}$. 

$E$ forms a lattice under the operations $\cap$ and $\cup$ defined
for $e,f\in E$ by $e\cap f(s)=\min(e(s),f(s))$ and $e\cup f(s)=\max(e(s),f(s))$.
The \emph{multiset difference} $e_{f}$ is $e-e\cap f$, and $e$
and $f$ are \emph{disjoint} if $e\cap f=e_{0}$. For instance $e_{f}$
and $f_{e}$ are disjoint. The \emph{symmetric difference} of $e$
and $f$ is $e\triangle f=e_{f}+f_{e}=e\cup f-e\cap f$. $e$ is a
\emph{submultiset} of $f$, written $e\subseteq f$, if $e(s)\leq f(s)\forall s$
and of course this is equivalent to $e\cap f=e$ or $e_{f}=e_{0}$. 

A \emph{function} $h$ from $e$ to $f$ is simply a function $h$
from $R(e)$ to $R(f)$, to guarantee that identical elements of $e$
are not mapped to distinct elements of $f$. We say that $h$ is an
\emph{injection} (resp. \emph{surjection}, \emph{bijection}), according
as (i) its restriction to the root sets has this property in the ordinary
sense, and (ii) for every $s\in R(e)$, $e(s)\leq f(h(s))$\emph{
(resp.} $e(s)\geq f(h(s))$, both of the preceding).

\subsection{Other metrics on multisets\label{sub:Other-metrics}}

We give a short account of the multiset metrics listed at \cite[pp.51-52]{DD09},
described elsewhere in that book, and regrouped here according to
the main idea.
\begin{itemize}
\item The \emph{matching distance}\cite[p.47]{DD09} is defined by $\inf_{g}\max_{x\in e}d(x,g(x))$
where $g$ runs over all (multiset) bijections from $e$ to $f$.
These are used in size theory (image recognition), where a geometric
trick is used to ensure that bijections are always defined. A survey
article is \cite{dFL06}.
\item The \emph{metric space of roots}\cite[p.221]{DD09} is defined on
multisets of $\mathbb{C}$ of fixed cardinality $n$, each identified
with the monic polynomial of which it is the set of roots. Two such
$u_{1},\ldots,u_{n}$ and $v_{1},\ldots,v_{n}$ are separated by $\min_{\rho}\max_{1\leq j\leq n}|u_{j}-v_{\rho(j)}|$
as $\rho$ ranges over the permutations of $1,\ldots,n$. More details
are in \cite{CM06}\emph{.}
\item Petrovsky has defined several metrics\cite[p.52]{DD09} on $E$ using
a measure $\mu:E\rightarrow\mathbb{R}$, $\mu(e)=\sum_{s\in S}\lambda(s)e(s)$
where $\lambda:S\rightarrow\mathbb{R}^{+}$. Thus $\mu=C$ when $\lambda=1$.
One of them is $d(e,f)=\mu(e\triangle f)=\mu(e_{f})+\mu(f_{e})$ and
the others are variants\cite{Pet97,Pet03}. They are related to the
Jaccard and Hamming metrics on sets\cite[p.299, p.45]{DD09}, and
seem to be primarily used in cluster analysis (decision making).
\item The $\mu$\emph{-metric}\cite[p.281]{DD09} on so-called phylogenetic
$X$-trees (computational biology), again is based on symmetric difference.
See \cite{10.1109/TCBB.2007.70270} for more details.
\item The \emph{bag distance}\cite[p.204]{DD09}\emph{,} used in string
matching, is defined to be $\max(C(e_{f}),C(f_{e}))$. 
\item In approximate string matching (for instance in bioinformatics), so-called
\emph{$q$-gram similarity}\cite[p.206]{DD09}\emph{ }is defined.
This is not a metric.
\end{itemize}
Note that there are two dominant ideas: minimising over multiset bijections,
and symmetric differences. The latter do not reflect any structure
on $S$ except perhaps if we argue that multiplicity may depend on
that structure. To some extent, the metrics described later mix these
two paradigms.

There are a number of other standard possibilities, such as the metric
induced on $E$ by any injection into a metric space, or those given
by taking the sum (or the supremum) of the $|e(s)-f(s)|$ where $e,f\in E$
and $s\in S$. Any metric on $\mathbb{Z^{+}}$ (multisets on the prime
numbers) is also an example.

\section{The multiset model $E$\label{sec:Model E}}

If $a,c\in E$ and $C(a)\leq C(c)$, we find a submultiset $c'$ of
$c$ of cardinality $C(a)$ so that, matching elements of $a$ and
$c'$ as described below, the sum of the $d$-distances is minimised,
and then we add a constant. The result, denoted $d_{E}$, though resembling
the matching distance just described, actually generalises the bag
distance. The other function, $d_{Em}$ ($m$ for 'mean') is obtained
by dividing $d_{E}$ by $C(c)$.

We choose $M>0$, and define $\theta=\frac{\sup d}{M}$ when $d$
is bounded. Given $a,c\in E$, suppose that $C(a)\leq C(c)$ and $c\neq e_{0}$.
Write down all the elements in both in arbitrary order, \emph{viz.,
$a_{1},a_{2},\ldots,a_{C(a)}$} and \emph{$c_{1},c_{2},\ldots,c_{C(c)}$
}where for each $x\in X$, $\#\{j:a_{j}=x\}=a(x)$ and $\#\{j:c_{j}=x\}=c(x)$.
(In other terminology, we \emph{parametrise} the multisets by enough
positive integers.)

Let $\gamma$ be a member of the permutation group $G_{c}$ on $C(c)$
elements, acting on the subscripts in the $c$-sequence. Write\emph{
}
\[
d^{\gamma}(a,c)=\sum_{j=1}^{j=C(a)}d(a_{j},c_{\gamma(j)})+M|C(c)-C(a)|
\]

and define the following functions $d_{E}$ and $d_{Em}$ from $E\times E$
to $R$. 

\[
d_{E}(a,c)=\min_{\gamma\in G_{c}}d^{\gamma}(a,c)\qquad\textrm{and}\qquad d_{Em}(a,c)=\frac{d_{E}(a,c)}{\max(C(a),C(c))}
\]

with $d_{Em}(e_{0},e_{0})=0$. We call $M|C(c)-C(a)|$ the \emph{notional
part} of $d_{E}(a,c)$. The mappings $\gamma$ regarded as from $a$
to $c$, need not be multiset functions.
\begin{prop}
If $d$ is unbounded then $d_{E}$ and $d_{Em}$ are non-metrics for
all $M$. If $d$ is bounded, then $d_{E}$ is a metric iff $\theta\leq2$,
and $d_{Em}$ is a metric iff $\theta\leq1$.\end{prop}
\begin{proof}
Only the triangle inequality need be verified or could fail. Let $x,y\in X$,
with $x\neq y$. Then 
\[
d_{E}(e_{x},e_{x}+e_{y})+d_{E}(e_{x}+e_{y},e_{y})-d_{E}(e_{x},e_{y})=2M-d(x,y)
\]
So if $d_{E}$ is a metric, $d$ is bounded and $\theta\leq2$. The
same argument for $d_{Em}$ implies $\theta\leq1$.

From now on we take $a,b,c\in E$, and assume $C(a)\leq C(c)$. We
look first at $d_{E}$ and suppose $\theta\leq2$: as motivation,
we could verify that whenever 
\[
2C(b)\leq C(a)(2-\theta)\quad\textrm{or}\quad2C(b)\geq\theta C(a)+2C(c)
\]
 then the notional parts alone in $d_{E}(a,b)+d_{E}(b,c)$ add to
at least 
\[
M(C(c)-C(a))+\theta MC(a)\geq d_{E}(a,c)
\]
The value of $C(b)$ determines three cases, all with similar reasoning.

Case $C(b)<C(a)$: there exist $\alpha\in G_{a}$ and $\gamma\in G_{c}$
such that

\[
d_{E}(a,b)=\sum_{i=1}^{C(b)}d(b_{i},a_{\alpha(i)})+M(C(a)-C(b))
\]
 and
\[
d_{E}(b,c)=\sum_{i=1}^{C(b)}d(b_{i},c_{\gamma(i)})+M(C(c)-C(b))
\]

So
\begin{align*}
d_{E}(a,b)+d_{E}(b,c) & \geq\sum_{i=1}^{C(b)}d(a_{\alpha(i)},c_{\gamma(i)})+M(C(c)-C(a))+2M(C(a)-C(b))
\end{align*}
\begin{alignat}{1}
 & \geq\sum_{i=1}^{C(a)}d(a_{\alpha(i)},c_{\gamma(i)})+M(C(c)-C(a))+(2-\theta)M(C(a)-C(b))\label{eq:B}
\end{alignat}

having added, for each $i$ beyond $C(b)$, the non-positive $d(a_{\alpha(i)},c_{\gamma(i)})-\theta M$.
Then \eqref{eq:B} is at least

\[
d_{E}(a,c)+M(2-\theta)(C(a)-C(b))\geq d_{E}(a,c)
\]

Case $C(a)\leq C(b)\leq C(c)$: there exist $\beta\in G_{b}$ and
$\gamma\in G_{c}$ such that 

\[
d_{E}(a,b)=\sum_{i=1}^{C(a)}d(a_{i},b_{\beta(i)})+M(C(b)-C(a))
\]

and
\[
d_{E}(b,c)=\sum_{i=1}^{C(b)}d(b_{i},c_{\gamma(i)})+M(C(c)-C(b))
\]

Then 
\begin{equation}
d_{E}(a,b)+d_{E}(b,c)=M(C(c)-C(a))+\sum_{i=1}^{C(a)}d(a_{i},b_{\beta(i)})+\sum_{i=1}^{C(b)}d(b_{i},c_{\gamma(i)})\label{eq:C}
\end{equation}

Now $\sum_{i=1}^{C(b)}d(b_{i},c_{\gamma(i)})=\sum_{i=1}^{C(b)}d(b_{\beta(i)},c_{\gamma\beta(i)})$
since $\beta\in G_{b}$, and so \eqref{eq:C} is at least

\[
M(C(c)-C(a))+\sum_{i=1}^{C(a)}d(a_{i},c_{\gamma\beta(i)})+\sum_{i=1+C(a)}^{C(b)}d(b_{\beta(i)},c_{\gamma\beta(i)})
\]

which is at least $d_{E}(a,c)$, in this case for any $\theta$.

Case $C(b)>C(c)$: For some $\tau\in G_{c}$, 
\[
d_{E}(a,c)=M(C(c)-C(a))+\sum_{i=1}^{C(a)}d(a_{i},c_{\tau(i)})
\]
and $\rho,\sigma\in G_{b}$ are given by

\[
d_{E}(a,b)=\sum_{i=1}^{C(a)}d(a_{i},b_{\rho(i)})+M(C(b)-C(a))
\]

and

\[
d_{E}(b,c)=\sum_{i=1}^{C(c)}d(b_{\sigma(i)},c_{i})+M(C(b)-C(c))
\]

We write $\omega=\sigma^{-1}\rho\in G_{b}$, which takes any subscript
of $a$ to a subscript of $c$, and define 
\[
l=\#\{1\leq i\leq C(a):\rho(i)=\sigma(j)\textrm{ for some }j\textrm{ in }1,\ldots,C(c)\}
\]
Then $l\geq C(a)+C(c)-C(b)$ since $\rho(1),\ldots,\rho(C(a))$ and
$\sigma(1),\ldots,\sigma(C(c))$ are all chosen from $1,2,\ldots,C(b)$.
Dropping all terms with $i>l$, $d_{E}(a,b)+d_{E}(b,c)$ is at least

\begin{equation}
\sum_{i=1}^{l}[d(a_{i},b_{\rho(i)})+d(b_{\rho(i)},c_{\omega(i)})]+2M(C(b)-C(c))+M(C(c)-C(a))\label{eq:A}
\end{equation}

Just as before, $d(a_{i},c_{\omega(i)})-\theta M\leq0$, so \eqref{eq:A}
is at least as big as

\[
\sum_{i=1}^{C(a)}d(a_{i},c_{\omega(i)})+2M(C(b)-C(c))-\theta M(C(a)-l)+M(C(c)-C(a))
\]

Now $C(b)-C(c)\geq C(a)-l\geq0$ so we get 

\[
d_{E}(a,b)+d_{E}(b,c)\geq\sum_{i=1}^{C(a)}d(a_{i},c_{\omega(i)})+M(C(c)-C(a))+M(2-\theta)(C(a)-l)\geq d_{E}(a,c)
\]

concluding the proof that $d_{E}$ is a metric.

Passing to $d_{Em}$, we now assume $\theta\leq1$, which implies
$d_{Em}(a,c)\leq M$. If $C(b)\leq C(c)$ it is certainly true that

\[
d_{Em}(a,b)+d_{Em}(b,c)\geq d_{Em}(a,c)
\]
so we will suppose $C(b)>C(c)$ and reuse the notation just employed
for $d_{E}$. Using \eqref{eq:A} again, we can write

\[
C(b)(d_{Em}(a,b)+d_{Em}(b,c))\geq(C(c)-l)M+\sum_{i=1}^{l}d(a_{i},c_{\omega(i)})+(2C(b)-2C(c)-C(a)+l)M
\]

Since $\theta\leq1$, the sum of the first two terms on the right
is at least $C(c)d_{Em}(a,c)$ and we also have $2C(b)-2C(c)-C(a)+l\geq C(b)-C(c)>0$,
so 

\[
d_{Em}(a,b)+d_{Em}(b,c)\geq\frac{C(c)d_{Em}(a,c)+M(C(b)-C(c))}{C(b)}\geq d_{Em}(a,c)
\]
as it is a convex combination of $d_{Em}(a,c)$ and $M$.
\end{proof}

\subsection{Simple properties of $d_{E}$ and $d_{Em}$}

We start with some computational results about $d_{E}$. The first
says that $a$ and $c$ can be taken as disjoint.
\begin{prop}
\label{pro:disj}$d_{E}(a,c)=d_{E}(a_{c},c_{a})$\end{prop}
\begin{proof}
Assume $C(a)\leq C(c)$. We have to show that among the permutations
$\gamma$ in $G_{c}$ which minimise 
\[
d^{\gamma}(a,c)=\sum_{j=1}^{C(a)}d(a_{j},c_{\gamma(j)})+M|C(c)-C(a)|
\]

there exists one in which maximally many identical elements (with
multiplicity) of $a$ and $c$ are matched up by $\gamma$. But if
$a_{j}=c_{\gamma(k)}$ then 
\[
d(a_{j},c_{\gamma(k)})+d(a_{k},c_{\gamma(j)})\leq d(a_{j},c_{\gamma(j)})+d(a_{k},c_{\gamma(k)})
\]
 is certainly true, so if we start with any $\gamma$ that minimises
$d^{\gamma}(a,c)$, we can find another with the required property.\end{proof}
\begin{cor}
If $\frac{d}{M}$ is the discrete metric then $d_{E}(a,c)=M\max(C(a_{c}),C(c_{a}))$,
and so $d_{E}$ generalises the bag distance.
\end{cor}
The next result is needed to establish completeness.
\begin{lem}
\label{squeeze} If $x,y\in X$, $a,c\in E$ and $C(a)=C(c)=n$, then

\[
|d_{E}(a+e_{x},c+e_{y})-d_{E}(a,c)|\leq d(x,y)
\]
\end{lem}
\begin{proof}
$d_{E}(a+e_{x},c+e_{y})\leq d(x,y)+d_{E}(a,c)$ because its right
side is obtained from its left side by permuting the subscripts in
the sense of the definition of $d_{E}$. 

Now, renumbering so as to identify $x$ as $a_{1}$ and $y$ as $c_{n+1}$
(if these subscripts were the same we would be finished) suppose that
\[
d_{E}(a+e_{x},c+e_{y})=d(x,c_{1})+d(a_{n+1},y)+\sum_{j=2}^{n}d(a_{j},c_{j})
\]
Then 
\[
d_{E}(a+e_{x},c+e_{y})+d(x,y)\geq d(a_{n+1},c_{1})+\sum_{j=2}^{n}d(a_{j},c_{j})\geq d_{E}(a,c)
\]

as required. Simple examples show that the bound $d(x,y)$ is tight.
\end{proof}
Finally we compare sequences in $d_{E}$ and $d_{Em}$.
\begin{prop}
\label{pro: sequence}Let $S_{i}$ be a sequence in $E$. Then any
of the following is true with respect to $d_{E}$ iff it is true with
respect to $d_{Em}$: (i) $S_{i}$ is Cauchy; (ii) $S_{i}$ is convergent;
(iii) $S_{i}$ has limit $l\in E$.\end{prop}
\begin{proof}
We first show that $d_{E}$ and $d_{Em}$ have the same Cauchy sequences.
Since multisets of cardinalities $r$ and $t$ are at least $M|t-r|$
apart in $d_{E}$, it follows that any Cauchy sequence for $d_{E}$
must eventually have constant cardinality, in which case $d_{E}$
and $d_{Em}$ are mutually proportional and so the sequence is also
Cauchy for $d_{Em}$.

Now suppose $S_{i}$ is Cauchy for $d_{Em}$, write $s_{i}=C(S_{i})$
and then for each $\epsilon>0,\exists N=N(\epsilon)$ such that whenever
$i,j>N$, $d_{Em}(S_{i},S_{j})<\epsilon$. But then $d_{Em}(S_{i},S_{j})\geq M(1-\frac{s_{j}}{s_{i}})$
supposing $s_{i}\geq s_{j}$ and as $\frac{s_{j}}{s_{i}}>1-\frac{\epsilon}{M}$
, no subsequence of the $s_{i}$ can go to infinity, and hence the
sequence $s_{i}$ is bounded (for each $s_{j}$, and hence in general).
But then $M(1-\frac{s_{j}}{s_{i}})$ only takes finitely many positive
values so for sufficiently small $\epsilon$ this gives a contradiction
unless the $s_{i}$ are eventually constant. So $d_{E}$ and $d_{Em}$
are again proportional and $S_{i}$ is Cauchy with respect to $d_{E}$.
(There is also a trivial case in which $s_{i}=0$ infinitely often.)

An exactly similar argument shows that any limit of such a sequence
(either metric) again has the same cardinality. It follows that $d_{E}$
and $d_{Em}$ also have the same convergent sequences (and limits).
\end{proof}
We are now ready for the main result.
\begin{prop}
(Topology and completeness.)

$d_{E}$ and $d_{Em}$ induce the same topology on $E$. The metrics
$d_{E}$ and $d_{Em}$ are complete iff $d$ is.\end{prop}
\begin{proof}
By \eqref{pro: sequence}, $d_{E}$ and $d_{Em}$ have the same convergent
sequences (and limits), and so induce the same topology on $E$. We
also see that given $d$, either both or neither of $d_{E}$ and $d_{Em}$
are complete metrics.

If $x_{i}$ is a non-convergent Cauchy sequence in $X$, then $S_{i}=\{x_{i}\}$
is a non-convergent Cauchy sequence for both $d_{E}$ and $d_{Em}$.

Supposing that $d$ is complete, let $S_{i}$ be a sequence of multisets
of $X$ which is Cauchy in $d_{E}$, with all $C(S_{i})=n>1$ (the
completeness of $d$ implies the case $n=1$). Given $\epsilon>0$,
$\exists N=N(\epsilon)$ such that $m\geq N\Longrightarrow d_{E}(S_{m},S_{N})<\epsilon$.
As every element%
\footnote{As always, this is with multiplicity. If some element of $X$ occurs
three times in $S_{N}$, then at least three elements (with multiplicity)
of each $S_{m}$ are within $d$-distance $\epsilon$ of it.%
} of each $S_{m}$ for $m\geq N$ is then within $d$-distance $\epsilon$
of some element of $S_{N}$ it follows that there exists a totally
bounded region of $X$ containing all elements of all the $S_{i}$.
Since $X$ is complete, the completion of this region is (can be regarded
as) a compact subset of $X$ and now we can assume that $X$ is compact.

Recalling that a Cauchy sequence converges iff it has a convergent
subsequence, we select an arbitrary $x_{i}$ from each $S_{i}$ (using
the axiom of choice). Since $X$ is compact, the sequence $x_{i}$
has a convergent subsequence $y_{i}=x_{t(i)}$ with limit $y$ (say).
Writing $T_{i}$ for $S_{t(i)}$, we denote by $T_{i}'$ the multiset
$T_{i}-e_{y_{_{i}}}$. Using \eqref{squeeze} we have

\[
|d_{E}(T_{i},T_{j})-d_{E}(T_{i}',T_{j}')|\leq d(y_{i},y_{j})
\]

and it follows that $T_{i}'$ is a Cauchy sequence of cardinality
$n-1$, and we can assume that $T_{i}'$ has limit $T'$. Using \eqref{squeeze}
again, and denoting $T'+e_{y}$ by $T$, 
\[
|d_{E}(T_{i},T)-d_{E}(T_{i}',T')|\leq d(y_{i},y)
\]

and so $T_{i}$ converges to $T$, which is therefore the limit of
the Cauchy sequence $S_{i}$.
\end{proof}

\subsection{An algorithm for $d_{E}$}

We show that calculation of $d_{E}$ is an integer programming problem.
As usual suppose $C(a)\leq C(c)$ and $a\cap c=e_{0}$. Just as in
the Hausdorff heuristic, label the rows (the columns) of a matrix
by elements of $R(a)$ (of $R(c)$), and put the $d$-distances as
entries. Add one more row whose entries are all $M$, to give a matrix
$D$.

Define a new matrix $H$, the same shape as $D$, constrained to satisfy
\[
\sum_{i}h_{ij}=c(j)\;\textrm{for}\; j\leq\#R(c)\text{\qquad\textrm{and}\qquad\ensuremath{\sum_{j}h_{ij}=a(i)\;\textrm{for}\; i\leq\#R(a)}}
\]
implying $\sum_{j}h_{1+\#R(a),j}=C(c)-C(a)$. Then $d_{E}(a,c)$ is
the minimum value of $\sum_{i,j}d_{ij}h_{ij}$ (the trace of $D^{T}H$),
for which all the $h_{ij}\in\mathbb{N}$.

\section{The multiset model $F$}

We will define a space $A$ whose finite subsets include the multisets
of $X$. 

This time we identify the multiset $re_{x}$ with $(x,r)\in X\times\mathbb{N}$,
as usual interpreted as {}``$r$ copies of $x$''. Let $A$ be the
quotient space of $X\times\mathbb{N}$ in which all points of the
form $(x,0)$ have been identified. $A_{r}$ will denote the (quotient
of the) subset $X\times\{r\}$. We use $\mathbb{N}$ instead of $\mathbb{Z}^{+}$
(which would be simpler) to get a canonical bijection with model $E$.
Note that $A$ consists of the isolated point $e_{0}$ and isolated
copies of $X$; furthermore $A$ coincides with $\{e\in E:\#R(e)\leq1\}$.

Hence a multiset of $X$ is a finite subset $U$ of $A$ whose underlying
elements of $X$ are all distinct, viz. $re_{x},se_{x}\in U\Longrightarrow r=s$
and $F$ will denote the space of all such subsets of $A$. The following
result should now be obvious.
\begin{prop}
Let $d'$ be any metric on $A$. Then the restriction of $d'_{H}$
to $F$ is a multiset metric on $X$, and it generalises the Hausdorff
metric iff $d'(1e_{x},1e_{y})=d(x,y)\forall x,y\in X$.
\end{prop}
We will return later to the question of when this is complete.

For metrics on $A$, as before fix $M>0$ and define $\theta=\frac{\sup d}{M}$
when $d$ is bounded. We start with the functions $d_{A}$ and $d_{Am}$
from $A\times A$ to $R$ defined by 
\[
d_{A}(re_{x},te_{z})=M|t-r|+\min(r,t)d(x,z)
\]

and 
\[
d_{Am}(re_{x},te_{z})=\frac{d_{A}(re_{x},te_{z})}{\max(r,t)}\quad\textrm{or\ensuremath{\quad0\quad}when\quad}r=t=0
\]

Noting that (a) these are well-defined on $A\times A$, (b) they are
the respective restrictions to $A\times A$ of $d_{E}$ and $d_{Em}$,
and (c) they both agree with $d$ when $r=t=1$, it follows that they
are metrics on $A$ when $\theta\leq2$ and when $\theta\leq1$ respectively.
Actually there is a small surprise.
\begin{prop}
If $d$ is unbounded then $d_{A}$ and $d_{Am}$ are non-metrics for
all $M$. If $d$ is bounded, then $d_{A}$ and $d_{Am}$ are both
metrics iff $\theta\leq2$. \end{prop}
\begin{proof}
As $d_{A}(2e_{x},e_{x})+d_{A}(e_{x},2e_{z})-d_{A}(2e_{x},2e_{z})=2M-d(x,z)$,
if $d_{A}$ is a metric, $d$ must be bounded and $\theta\leq2$.
Use the same example for $d_{Am}.$

It only remains to show that $d_{Am}$ is a metric when $\theta\leq2$.
We fix $re_{x},te_{z}\in A$, assuming $r\leq t$. Now if $s\leq t$,
it is immediate that

\[
d_{Am}(re_{x},se_{y})+d_{Am}(se_{y},te_{z})\geq d_{Am}(re_{x},te_{z})
\]

so we will take $s>t$. Using the definition of $d_{Am}$, 
\begin{equation}
st(d_{Am}(re_{x},se_{y})+d_{Am}(se_{y},te_{z})-d_{Am}(re_{x},te_{z}))\label{eq:D}
\end{equation}

\[
=M(t+r)(s-t)+rtd(x,y)+t^{2}d(y,z)-rsd(x,z)
\]

and using $t^{2}\geq rt$ we get that \eqref{eq:D} is at least as
large as

\[
Mt(s-t)+r(s-t)(M-d(x,z))
\]

which is non-negative provided $2M\geq\frac{2r}{r+t}d(x,z)$, whose
right side cannot exceed $\sup d=\theta M$. So $d_{Am}$ is a metric
when $\theta\leq2$.\end{proof}
\begin{rem}
If $r\leq t$, $M(t-r)\leq td_{Am}(re_{x},te_{z})=d_{A}(re_{x},te_{z})\leq tM\max(1,\theta)$.
Actually, $d_{Am}(re_{x},te_{z})$ is a convex combination of $d(x,z)$
and $M$ and therefore lies between them. \end{rem}
\begin{prop}
Let $r_{i}e_{x(i)}$ be a sequence in $A$. Then any of the following
is true with respect to $d_{A}$ iff it is true with respect to $d_{Am}$:
(i) $r_{i}e_{x(i)}$ is Cauchy; (ii) $r_{i}e_{x(i)}$ is convergent;
(iii) $r_{i}e_{x(i)}$ has limit $l\in A$.
\end{prop}
The proof is exactly as in \eqref{pro: sequence}.
\begin{prop}
(Main properties of $A$)\end{prop}
\begin{enumerate}
\item $d_{Am}$ and $d_{A}$ both induce the same topology on $A$, coinciding
with the quotient topology inherited from $X\times\mathbb{N}$.
\item $d_{A}$ and $d_{Am}$ are complete metrics iff $d$ is.
\item The subset $U$ of $A$ is compact iff each $U_{r}=U\cap A_{r}$ is
a compact subset of $A_{r}$, and almost all the $U_{r}$ are empty.\end{enumerate}
\begin{proof}
(Clause 1) We have just seen that $d_{A}$ and $d_{Am}$ have the
same convergent sequences and limits, so they induce the same topology.

Let $re_{x},te_{z}\in A$ with $t>0$ and choose $\epsilon>0$. Now
$M|r-t|\leq d_{A}(re_{x},te_{z})<\epsilon$ implies $r=t$ when $\epsilon$
is sufficiently small, and indeed in this case
\[
d_{A}(te_{x},te_{z})<\epsilon\Leftrightarrow d(x,z)<\frac{\epsilon}{t}
\]

It follows that any sufficiently small open ball around $te_{z}$
in the $d_{A}$-topology is also an open ball in the quotient topology,
and vice versa. 

The point $e_{0}$ is isolated in both topologies. So these three
topologies on $A$ coincide.

(Clause 2) By the preceding proposition $d_{A}$ is complete iff $d_{Am}$
is. As any Cauchy sequence eventually lies in a single $A_{r}=X\times\{r\}$,
it converges iff this is true for the same sequence regarded as a
sequence in $X$, and any limits also coincide.

(Clause 3) Suppose $U$ is compact. If infinitely many $U_{r}$ were
non-empty we could find a sequence in $U$ with no convergent subsequence
(compactness being equivalent to sequential compactness in metric
spaces). If $re_{x(i)}$ is a sequence in some $U_{r}$ then it has
a convergent subsequence in $U$ but this must converge to a point
of $U_{r}$. Conversely, if $U_{r}$ is compact in $A_{r}$ then it
is compact in $X$ and then $U$ is a finite union of compact sets,
and so compact.
\end{proof}
Let $d_{F}$ and $d_{Fm}$ be the Hausdorff metrics arising from $d_{A}$
and $d_{Am}$ respectively. Let us write $F'$ for the set of all
finite subsets of $A$.
\begin{prop}
$d_{F}$ and $d_{Fm}$ are metrics on $F'$ iff $\theta\leq2$, and
both coincide with the Hausdorff metric for the case of ordinary subsets.
They are complete metrics on $F'$ iff $d$ is.\end{prop}
\begin{proof}
$\theta\leq2$ is necessary for the triangle inequality for $d_{A}$
(and so for $d_{F}$) or for $d_{Am}$ (and so for $d_{Fm}$) to hold.
The rest of the statement is an immediate consequence of their definitions
and the stated properties of the Hausdorff metric.
\end{proof}

\section{The multiset model $G$}

We continue to suppose $\theta\leq2$. Of course the restrictions
of $d_{F}$ and $d_{Fm}$ to $F$ (multisets on $X$) need not be
complete. For instance, if $d(x_{i},x)$ is strictly decreasing to
zero and $y_{i}=x_{i+1}$, then the sequence $2e_{x_{i}}+3e_{y_{i}}$
is Cauchy in $d_{F}$ or $d_{Fm}$ but its limit is $\{2e_{x},3e_{x}\}\in F'\backslash F$.

We deal with this discrepancy in the following way. Observe that to
every $U\in F'$ there is a function $t_{U}:X\rightarrow\mathbb{N}$
defined by 
\[
t_{U}(x)=\sum_{ae_{x}\in U}a
\]
and indeed if $U\in F$, $t_{U}$ is its representative in model $E$.
Define an equivalence relation $\sim$ on $F'$ by decreeing $U\sim V$
iff $t_{U}=t_{V}$. For example, if $x\neq y$ one $\sim$-class is
$\{e_{x},2e_{x},3e_{x},2e_{y}\},\{e_{x},5e_{x},2e_{y}\},\{2e_{x},4e_{x},2e_{y}\},\{6e_{x},2e_{y}\}$.
Obviously every class is finite, contains exactly one element of $F$,
and is a singleton iff $t_{U}(x)\leq2\forall x$. 

We now write $G$ for $F'/\!\sim$ and $d_{G}$ for the quotient pseudometric
on $G$ corresponding to $d_{F}$. There are canonical bijections
among $G$, $F$ and $E$. We extend the notations $e_{0}$, $R()$
and $C()$ to $F'$ and $G$ in the obvious way. If $e\in G\backslash\{e_{0}\}$,
it follows that $d_{G}(e,e_{0})\geq M$ since $d_{A}(re_{x},e_{0})\geq M$
for all $re_{x}\in A,r\neq0$.

Now $d_{G}$ is definitely less than $d_{F}$ in general as

\[
d_{G}(3e_{x},3e_{y})\leq d_{F}(\{e_{x},2e_{x}\},\{e_{y},2e_{y}\})\leq2d(x,y)<3d(x,y)=d_{F}(3e_{x},3e_{y})
\]

The most important facts about $d_{G}$ are corollaries of the following
result.
\begin{prop}
\label{prop:dGdFdH}If $e,f\in G\backslash\{e_{0}\}$, then $d_{G}(e,f)\geq d_{H}(R(e),R(f))$.\end{prop}
\begin{proof}
Suppose $x\in R(e),y\in R(f)$ are such that $d(x,y)=d_{H}(R(e),R(f))$.
We can assume $x\notin R(f)$. Let $e=p_{0},p_{1},\ldots,p_{n},p_{n+1}=f$
be a sequence of elements of $G$, referring to the notation of \eqref{quotient}.
If any $p_{i}$ is $e_{0}$ then we have two or more terms $\geq M$
so the path length is at least $2M\geq\sup d\geq d(x,y)$ and we now
assume that all $R(j):=R(p_{j})$ are non-empty.

We will employ the observation that $d_{F}(u,v)\geq\min_{b\in R(v)}d(a,b)$
if $a\notin R(v)$. For any sequence $x_{0},x_{1},\ldots$, all in
$\cup_{j}R_{j}$, define $s_{i}$ by $x_{i}\in R(s_{i})$ where $s_{i}$
is maximal. Take $x_{0}=x$ and choose $x_{1}\in R(1+s_{0})$ such
that $d(x_{0},x_{1})$ is minimal. So the $d_{F}$-distance between
any member of $p_{s_{0}}$ and any member of $p_{1+s_{0}}$ is at
least $d(x_{0},x_{1})$.

If $x_{1}\in R(f)$ we are finished as our path is at least $d(x_{0},x_{1})\geq d(x,y)$.
Otherwise choose $x_{2}\in R(1+s_{1})$ such that $d(x_{1},x_{2})$
is minimal.

Again we are finished if $x_{2}\in R(f)$ as our path is (at least)
$d(x_{0},x_{1})+d(x_{1},x_{2})$. If not, choose $x_{3}\in R(1+s_{2})$
to minimise $d(x_{2},x_{3})$. As the $s_{i}$ are increasing we get
a sequence of terms from $x$ to some $z\in R(f)$ whose sum is at
least $d(x,y)$.\end{proof}
\begin{cor}
(1) $d_{G}$ agrees with the Hausdorff metric on finite subsets of
$X$.

(2) If $d$ is uniformly discrete then $d_{G}$ is a complete metric
on $G$.

(3) $d_{F}(e,f)\geq d_{H}(R(e),R(f))$.\end{cor}
\begin{proof}
(1) For finite subsets $e,f$ of $X$, 
\[
d_{F}(e,f)\geq d_{G}(e,f)\geq d_{H}(R(e),R(f))=d_{H}(e,f)=d_{F}(e,f)
\]

(2) $d_{H}$ has the same lower bound as $d$. By clause (1), so does
$d_{G}$, making it a metric. $d_{G}$ is complete because it is uniformly
discrete.

(3) $d_{F}$ is at least as big as $d_{G}$.
\end{proof}
In the notation of the proposition, if we have $t_{e}(x)>t_{f}(x)$
and we define $s_{0}$ to be the maximal $s$ such that $t_{s}(x)>t_{1+s}(x)$,
we cannot use the same argument to show that $d_{G}$ is a metric
in general, because we might have $z=x$.

\section{Concluding remarks}

Aside from the potential applications mentioned at the start or described
in \cite{SIYS07}, these metrics might also be useful in voting theory.
An election is a multiset on the set $X$ of permitted ballot types.
For instance, if $X$ is the total orderings (permutations) of $n$
candidates, one well-known metric on $X$ is the \emph{Kendall $\tau$-distance}\cite[p.211]{DD09},
defined as the fewest transpositions required to change one into the
other. 

Future work ought to look at possible applications and clarify the
relationships among $E,F$ and $G$.

\bibliographystyle{amsalpha}
\bibliography{MultisetMetricsOnBoundedSpaces}

\end{document}